\documentclass[10pt]{amsart}
\usepackage{xypic,somedefs,tracefnt,latexsym,rawfonts,epsfig,amsfonts,amssymb,latexsym,enumerate,amscd}
\def\cal{\mathcal}
\def\Bbb{\mathbb}

\def\vtr{\trianglelefteq}
\makeatletter
\@namedef{subjclassname@2020}{%
  \textup{2020} Mathematics Subject Classification}
\makeatother
\newtheorem{thm}{Theorem}
\newtheorem{prop}[thm]{Proposition}

\newtheorem{cor}[thm]{Corollary}
\newtheorem{defn}[thm]{Definition}

\numberwithin{equation}{section}
\begin{document}
\title[Some virtually poly-free Artin groups]{Some
  virtually poly-free Artin groups}
\author[S.K. Roushon]{S.K. Roushon}
\address{School of Mathematics\\
Tata Institute\\
Homi Bhabha Road\\
Mumbai 400005, India}
\email{roushon@math.tifr.res.in} 
\urladdr{http://www.math.tifr.res.in/\~\ roushon/}
\thanks{July 05, 2020}
\begin{abstract} In this short note we prove that a class of
  Artin groups of affine and complex types are virtually poly-free, answering 
  partially the question if all Artin groups are
  virtually poly-free.
 \end{abstract}
 

\maketitle

We recall the following definition.

\begin{defn}\label{VPF} {\rm Let $\cal F$ and $\cal {VF}$ denote the class of free groups
    and virtually-free groups, respectively. Let $\cal C$ be either $\cal F$ or $\cal {VF}$.
    A group $G$ is
    called {\it virtually poly}-$\cal C$, if $G$ contains a finite index subgroup $H$, and $H$
    admits a normal series $1=H_0\vtr H_1\vtr H_2\vtr \cdots \vtr H_n=H$, such that 
    $H_{i+1}/H_i\in {\cal C}$, for $i=0,1,\ldots, n-1$. In this case, $H$ is called {\it poly-$\cal C$} or that
    $H$ has a {\it poly-$\cal C$\ structure}. The minimum such $n$ is called the
  {\it length} of the poly-$\cal C$ structure.}\end{defn}

Poly-$\cal F$ groups have nice properties like, locally indicable and right orderable.
In [\cite{MB}, Question 2] it was asked if all Artin groups are virtually
poly-$\cal F$. Among the finite type Artin groups, 
the groups of types $A_n$, $B_n (=C_n)$, $D_n$,
$F_4$, $G_2$ and $I_2(p)$ are already known to be virtually poly-$\cal F$ (\cite{Br}).

Here, we extend this class and prove the following theorem.

\begin{thm}\label{mt} Let $\cal A$ be an Artin group of the affine type
  $\tilde A_n$, $\tilde B_n$, $\tilde C_n$, $\tilde D_n$ or of the finite
  complex type $G(de,e,r)$ ($d,r\geq 2$).
  Then, $\cal A$ is virtually poly-$\cal F$.\end{thm} 
  
The main idea behind the proof of Theorem \ref{mt} is the following result,
which is easily deducible from [\cite{R}, Theorem 2.2 and Remark 2.4]. 

Let ${\Bbb C}(m,k;q)$ be the orbifold, whose underlying space is the complex plane minus $m$ points
$p_1,p_2,\ldots , p_m\in {\Bbb C}$, with $k$ cone points
$x_1,x_2,\ldots x_k\in {\Bbb C}-\{p_1,p_2,\ldots , p_m\}$ of orders $q_1,q_2\ldots, q_k$,
respectively. $q$ denotes the $k$-tuple $(q_1,q_2,\ldots, q_k)$. Let $PB_n({\Bbb C}(m,k;q))$ be the configuration orbifold of $n$ distinct
points of ${\Bbb C}(m,k;q)$. By convention $PB_1({\Bbb C}(m,k;q))={\Bbb C}(m,k;q)$.

\begin{thm}\label{MR} The orbifold fundamental group $\pi_1^{orb}(PB_n({\Bbb C}(m,k;q)))$
  has a poly-$\cal {VF}$ structure, consisting of finitely
  presented subgroups in a normal series.\end{thm}

\begin{proof} Recall that, in [\cite{R}, Theorem 2.2 and Remark 2.4] we proved
  the following exact sequence. 
  The second homomorphism is induced by the projection to the first $n-1$ coordinates.

\centerline{
\xymatrix{1\ar[r]&K\ar[r]&\pi_1^{orb}(PB_n(S))\ar[r]&\pi_1^{orb}(PB_{n-1}(S))\ar[r]&1.}}
Here, $S={\Bbb C}(k,m; q)$, and $K$ is isomorphic to $\pi_1^{orb}(F)$,
$F=S-\{(n-1)-\text{regular points}\}$. By {\it regular points}
we mean points which are not cone points, that is, 
points in ${\Bbb C}-\{x_1,\ldots, x_m,p_1,\ldots, p_k\}$. That is, $F={\Bbb C}(k, m+n-1;q)$.

The theorem now follows by induction on $n$, since $\pi_1^{orb}({\Bbb C}(k,m;q))$ is finitely presented and 
virtually free, for all $k$, $m$ and $q$.
\end{proof}

Note that the symmetric group $S_n$ acts on $PB_n({\Bbb C}(m,k;q))$ by permuting the 
coordinates. Hence, 
the quotient $PB_n({\Bbb C}(m,k;q))/S_n$ is again an orbifold, and it is denoted by
$B_n({\Bbb C}(m,k;q))$.

Therefore, by Theorem \ref{MR} we have the following corollary, since if a
group has a finite index normal poly-$\cal {VF}$ subgroup, then the group is also
poly-$\cal {VF}$.

\begin{cor}\label{MRC} The orbifold
  fundamental group $\pi_1^{orb}(B_n({\Bbb C}(m,k;q)))$ is finitely presented, and has a poly-$\cal {VF}$
  structure, consisting of finitely
  presented subgroups in a normal series.\end{cor}

To prove our main theorem, furthermore, we need the following two results.

\begin{thm}\label{AO} All affine type Artin groups are torsion free.\end{thm}

\begin{proof} This was recently proved in \cite{PS}.\end{proof}

\begin{thm}(\cite{All}) \label{All} Let $\cal A$ be an Artin group, and $\cal O$ be an orbifold as 
described in the following table. Then, $\cal A$ can be embedded as a normal 
subgroup in $\pi_1^{orb}(B_n({\cal O}))$. The third column gives the quotient
group $\pi_1^{orb}(B_n({\cal O}))/{\cal A}$.

\medskip
\centerline{
\begin{tabular}{|l|l|l|l|}
\hline
{\bf Artin group of type} & {\bf Orbifold $\cal O$}&{\bf Quotient group}&$n$\\
\hline \hline
&&&\\
$B_n$& ${\Bbb C}(1,0)$ &$<1>$& $n>1$\\
\hline
&&&\\
$\tilde A_{n-1}$& ${\Bbb C}(1,0)$ &${\Bbb Z}$& $n>2$\\
\hline
&&&\\
$\tilde B_n$&${\Bbb C}(1,1;(2))$  &${\Bbb Z}/2$& $n>2$\\
\hline
&&&\\
$\tilde C_n$&${\Bbb C}(2,0)$ &$<1>$&$n>1$\\
  \hline
  &&&\\
  $\tilde D_n$&${\Bbb C}(0,2;(2,2))$&${\Bbb Z}_2\times{\Bbb Z}_2$&$n>2$\\
  \hline
\end{tabular}}

\medskip
\centerline{\rm{Table}}
\end{thm}

\begin{proof} See \cite{All}.\end{proof}
  
We also need the following.

\begin{prop}\label{fi} If a torsion free, finitely presented poly-$\cal {VF}$ group
  has a normal series with finitely presented subgroups, then the group is
  virtually poly-$\cal F$.\end{prop}

\begin{proof} Let $H$ be a poly-$\cal {VF}$ group of length $n$ satisfying the 
  hypothesis of the statement.
  The proof is by induction on $n$. If $n=1$, then $H$ is virtually free, and hence free, since it is
  torsion free (\cite{S}). Therefore, assume that the lemma is true for all 
  poly-$\cal {VF}$ groups of length $\leq n-1$ satisfying the hypothesis. Consider a
  finitely presented normal series for $H$ giving the poly-$\cal {VF}$ structure.
  Then, $H_{n-1}$ is finitely presented, torsion free and has a poly-$\cal {VF}$
  structure of length $n-1$. Hence, by the induction hypothesis, there is a finite index
  subgroup $K\leq H_{n-1}$ and $K$ is poly-$\cal F$. Since $H_{n-1}$ is finitely
  presented, we can find a finite index subgroup $K'$ of $K$ which is 
  also a characteristic subgroup of $H_{n-1}$. Hence, $K'$ is also a poly-$\cal F$
  group, and is a normal subgroup of $H=H_n$ with quotient virtually
  free. Let $q:H\to H/K'$ be the quotient map. Consider a free subgroup $L$ of $H/K'$ of finite
  index, then $q^{-1}(L)$ is a finite index poly-$\cal F$ subgroup of $H$.

  This proves the Proposition.\end{proof}

Now, we are ready to prove our main theorem.

\begin{proof}[Proof of Theorem \ref{mt}] From the Table of Theorem \ref{All}, we see
  that the Artin group of type
  $\tilde A_n$ is a subgroup of the finite type Artin group of type $B_{n+1}$. Hence,
  by \cite{Br}, the Artin group of type $\tilde A_n$ is virtually poly-$\cal F$.
  
  Now, let $\cal A$ be an Artin group of type $\tilde B_n$, $\tilde C_n$ or $\tilde D_n$.
  Then, by Theorem \ref{All}, $\cal A$ can be embedded as a 
  normal subgroup in $\pi_1^{orb}(B_n({\Bbb C}(k,m;q)))$ of finite index, for some suitable $k,m,q$ and $n$.
  Next, note that by Corollary \ref{MRC}, $\pi_1^{orb}(B_n({\Bbb C}(k,m;q)))$ is poly-$\cal {VF}$ by
  a normal series consisting of finitely presented subgroups. Since
  $\cal A$ is finitely presented and of finite index in $\pi_1^{orb}(B_n({\Bbb C}(k,m;q)))$, it follows that 
  $\cal A$ is also poly-$\cal {VF}$, by a normal series consisting of finitely presented
  subgroups. But by Theorem \ref{AO}, $\cal A$ is
  also torsion free. Hence, by Proposition \ref{fi} $\cal A$ is virtually poly-$\cal F$.

  The $G(de,e,r)$ ($d,r\geq 2$) type case is easily deduced from the fact that, 
  this Artin group can be embedded as a subgroup in the finite type Artin group of type
  $B_r$. See [\cite{CLL}, Proposition 4.1].
  
Therefore, we have completed the proof of Theorem \ref{mt}.\end{proof}

\newpage
\bibliographystyle{plain}
\ifx\undefined\bysame
\newcommand{\bysame}{\leavevmode\hbox to3em{\hrulefill}\,}
\fi

\end{document}